\documentclass{lms}
\usepackage{amsmath}
\usepackage{amssymb}
\usepackage{url}
\usepackage{enumerate}

\newtheorem{thm}{Theorem}
\newtheorem{prop}[thm]{Proposition}
\newtheorem{lem}[thm]{Lemma}

\numberwithin{equation}{section}
\numberwithin{thm}{section}

\newcommand{\ZZ}{\mathbf{Z}}
\newcommand{\QQ}{\mathbf{Q}}
\newcommand{\CC}{\mathbf{C}}
\newcommand{\FF}{\mathbf{F}}
\newcommand{\NN}{\mathbf{N}}
\newcommand{\TT}{\mathbf{T}}
\newcommand{\PP}{\mathbf{P}}
\newcommand{\RR}{\mathbf{R}}

\newcommand{\LL}{\mathcal{L}}
\newcommand{\norm}[1]{\Vert #1 \Vert}
\newcommand{\divides}{\mathrel |}
\newcommand{\ndivides}{\mathrel \nmid}
\newcommand{\eps}{\varepsilon}
\newcommand{\red}[1]{\bar{#1}}
\DeclareMathOperator{\trace}{tr}
\DeclareMathOperator{\Proj}{Proj}
\DeclareMathOperator{\Spec}{Spec}

\title{Computing zeta functions of arithmetic schemes}
\author{David Harvey}

\classno{11G40 (Primary) 11M38, 11Y16, 14G10 (Secondary)}

\begin{document}

\maketitle

\begin{abstract}
We present new algorithms for computing zeta functions of algebraic varieties over finite fields.
In particular, let~$X$ be an arithmetic scheme (scheme of finite type over~$\ZZ$), and for a prime~$p$ let $\zeta_{X_p}(s)$ be the local factor of its zeta function.
We present an algorithm that computes $\zeta_{X_p}(s)$ for a single prime~$p$ in time $p^{1/2+o(1)}$, and another algorithm that computes $\zeta_{X_p}(s)$ for all primes $p < N$ in time $N \log^{3+o(1)} N$.
These generalise previous results of the author from hyperelliptic curves to completely arbitrary varieties.
\end{abstract}

\section{Introduction}
\label{sec:intro}

Let~$X$ be an arithmetic scheme, i.e., a scheme of finite type over~$\ZZ$.
Loosely speaking, this is an object defined locally by polynomial equations in finitely many variables over~$\ZZ$.
The zeta function of such a scheme, introduced by Serre~\cite{Ser-zeta-L}, is defined by
 \[ \zeta_X(s) = \prod_x \frac{1}{1 - N(x)^{-s}}, \]
where the product is taken over all closed points $\{x\}$ of~$X$, and where $N(x)$ denotes the cardinality of the residue field of~$x$.
It admits an Euler product
 \[ \zeta_X(s) = \prod_p \zeta_{X_p}(s) \]
where $X_p = X \times_{\ZZ} \ZZ/p\ZZ$ is the reduction of~$X$ modulo~$p$ for each prime~$p$.
The local factors have the form $\zeta_{X_p}(s) = Z_{X_p}(p^{-s})$ where $Z_{X_p}(T) \in 1 + T \ZZ[[T]]$ is a rational function of~$T$.
To compute $\zeta_{X_p}(s)$ means to find the numerator and denominator of $Z_{X_p}(T)$ as polynomials in $\ZZ[T]$.
The main result of this paper is the following.
\begin{thm}
\label{thm:scheme}
Let $X$ be an arithmetic scheme.
\begin{enumerate}[(a)]
\item There exists a deterministic algorithm that takes as input a prime~$p$, and outputs $\zeta_{X_p}(s)$.
It has time complexity $p \log^{1+\eps} p$ and space complexity $O(\log p)$.
\item There exists a deterministic algorithm that takes as input a prime~$p$, and outputs $\zeta_{X_p}(s)$.
It has time complexity $p^{1/2} \log^{2+\eps} p$ and space complexity $O(p^{1/2} \log p)$.
\item There exists a deterministic algorithm that takes as input an integer $N \geq 2$, and outputs $\zeta_{X_p}(s)$ for all primes $p < N$.
It has time complexity $N \log^{3+\eps} N$ and space complexity $O(N \log^2 N)$.
\end{enumerate}
\end{thm}

By ``time complexity'' we mean the number of bit operations, or more precisely, the number of steps on a multi-tape Turing machine, in the sense of~\cite{Pap-complexity}.
The notation $z^{\eps}$ means a function bounded by $C z^{g(z)}$, where $C > 0$ is an absolute constant and where $g(z)$ is a non-negative function such that $\lim_{z\to\infty} g(z) = 0$.
By ``space complexity'' we mean the number of distinct tape cells visited during a computation.
This includes the space occupied by the output, but not the input, which may be assumed read-only.
The space complexity of an algorithm never exceeds its time complexity.
In Theorem~\ref{thm:scheme}, the implied constants depend of course on~$X$.

The complexity bounds in Theorem~\ref{thm:scheme} improve substantially on previously known algorithms.
We discuss this in more detail below, but for the moment we point out that for a variety in affine $n$-space, the previous best result was that of Lauder and Wan~\cite{LW-counting}, who obtained the bound $p^{O(n)}$ for computing a single $Z_{X_p}(T)$.
Theorem~\ref{thm:scheme}(a) reduces the exponent of~$p$ in the time complexity from $O(n)$ to $1 + \eps$, regardless of the dimension, and the space complexity is the minimum conceivable.
Part (b) reduces the exponent of~$p$ further to $1/2 + \eps$, but gives up the gains in space.
Part (c) reduces the complexity from exponential to \emph{polynomial} in $\log p$, provided we average over $p < N$.
The average complexity per prime is $\log^{4+\eps} N$, where again the exponent $4 + \eps$ is independent of the dimension.
Even for the simplest nontrivial case of an elliptic curve, this bound is competitive with the best known variants of Schoof's algorithm, both deterministic and probabilistic (see~\cite{Har-avgpoly} for further discussion).

The algorithms introduced in this paper are ``elementary'' in the sense that they do not rely on any cohomology theory, either $p$-adic or $\ell$-adic.
We have not yet implemented them on a computer, and we do not know if they are practical.
However, we remark that in recent joint work with Sutherland~\cite{HS-hassewitt,HS-hassewitt2}, an algorithm similar to that of (c) has been implemented for the special case of a hyperelliptic curve of genus $2$ or $3$; for $N$ around $2^{30}$ it outperforms previous methods by a factor of more than 300.

The bulk of the paper is devoted to the special case of a hypersurface in an affine torus.
Our results for this case are more concrete and more precise than those for a general arithmetic scheme.
In the remainder of this section, we state the hypersurface results, discuss previous algorithms in the literature, and show how to deduce the general case from the hypersurface results.

\subsection{Hypersurfaces over finite fields}

Let~$p$ be a prime, let~$a$ be a positive integer, and let $\FF_q$ be the finite field with $q = p^a$ elements.
For $n \geq 1$ let $\PP^n_{\FF_q}$ denote projective $n$-space over $\FF_q$ with homogeneous coordinates $x_0, \ldots, x_n$, and let $\TT^n_{\FF_q}$ be the affine torus $\{x_0 \cdots x_n \neq 0\} \subset \PP^n_{\FF_q}$.
Let $\red F \in \FF_q[x_0, \ldots, x_n]$ be homogeneous of degree $d \geq 1$, and let~$X$ be the hypersurface in $\TT^n_{\FF_q}$ cut out by~$\red F$.
We allow the case $\red F = 0$, in which case $X = \TT^n_{\FF_q}$.
The zeta function of~$X$ is $\zeta_X(s) = Z_X(q^{-s})$ where
 \[ Z_X(T) = \exp \left(\sum_{r \geq 1} \frac{|X(\FF_{q^r})|}r T^r \right) \in 1 + T\ZZ[[T]]. \]
Very explicitly, $|X(\FF_{q^r})|$ is given by
 \[ |X(\FF_{q^r})| = \frac{1}{q^r - 1} \big|\big\{(c_0, \ldots, c_n) \in (\FF_{q^r}^*)^{n+1} : \red F(c_0, \ldots, c_n) = 0 \big\}\big|. \]
For such~$X$, we have the following more precise versions of Theorem~\ref{thm:scheme}(a) and (b).
\begin{thm}
\label{thm:hs-linear}
There exists an explicit deterministic algorithm with the following properties.
The input consists of positive integers~$a$, $n$, $d$, a prime~$p$ not dividing~$d$, a monic irreducible polynomial $\red f \in \FF_p[t]$ of degree~$a$ defining the finite field $\FF_q \cong \FF_p[t]/\red f$, and a homogeneous polynomial $\red F \in \FF_q[x_0, \ldots, x_n]$ of degree~$d$, defining a hypersurface~$X$ in $\TT^n_{\FF_q}$ as above.
The output is $Z_X(T)$.
The algorithm has time complexity
 \[ 2^{6n^2 + 13n} n^{3n+4+\eps} (d+1)^{3n^2+6n+\eps} a^{3n+4+\eps} p \log^{1+\eps} p \]
and space complexity
 \[ O(2^{4n^2+9n} n^{2n+2} (d+1)^{2n^2+4n} a^{2n+3} \log p). \]
\end{thm}
\begin{thm}
\label{thm:hs-sqrt}
There exists an explicit deterministic algorithm with the following properties.
The input and output is the same as in Theorem~\ref{thm:hs-linear}.
The algorithm has time complexity
 \[ 2^{8n^2 + 16n} n^{4n+4+\eps} (d+1)^{4n^2+7n+\eps} a^{4n+4+\eps} p^{1/2} \log^{2+\eps} p \]
and space complexity
 \[ O(2^{4n^2+9n} n^{2n+2} (d+1)^{2n^2+4n} a^{2n+3} p^{1/2} \log p). \]
\end{thm}

Each element of $\FF_q$ occupies $O(a \log p)$ bits, so the total size of the input in Theorems~\ref{thm:hs-linear} and~\ref{thm:hs-sqrt} is $O((d+1)^n a \log p)$.
The requirement that $p \ndivides d$ is not a serious restriction: if $p \divides d$ we may simply replace~$\red F$ by $x_0 \red F$, increasing the degree by one, without changing $Z_X(T)$. 

Let us compare these complexity bounds to known algorithms in the literature.
We will mainly emphasise the dependence of the complexity on~$p$, as this is where the new algorithms have a decisive advantage.

First consider the naive algorithm: brute force enumeration of points.
To compute $|X(\FF_{q^r})|$ for a given~$r$, this requires roughly~$q^{nr}$ function evaluations.
To obtain the whole zeta function we must compute $|X(\FF_{q^r})|$ for~$r$ up to $2(4d+4)^n$ (see Lemma~\ref{lem:bd-zeta1}).
This leads to an overall complexity bound somewhat worse than~$q^{d^n}$, which is exponential in $\log p$, $a$, and~$d^n$.
On the positive side, the naive algorithm is very economical with respect to space!

The best result for an arbitrary hypersurface was previously that of Lauder and Wan~\cite{LW-counting}, who proved that $Z_X(T)$ may be computed in time $(d^n a p)^{O(n)}$.
In particular, for fixed~$p$ and~$n$, their bound is polynomial in the input size.

Recently, Sperber and Voight~\cite{SV-fewnomial} used a related method to improve the dependence on~$p$ in the case that the defining polynomial is \emph{sparse} and \emph{nondegenerate}.
The sparsity affects the exponent of~$p$ in the complexity bound: if~$\red F$ has sufficiently few monomials it is only~$p^{1+\eps}$, but it rises to $p^{2n+\eps}$ in the dense case.
The nondegeneracy condition is a slightly stronger condition than smoothness.
Voight has also pointed out (personal communication) that in the maximally sparse case their method may be modified to obtain results of the same strength as parts (b) and (c) of Theorem~\ref{thm:scheme}.

Both of the abovementioned papers actually give more precise results taking into account the Newton polytope of the defining polynomial.
Roughly speaking, one replaces~$d^n$ in the complexity bounds by the volume of the polytope.
It is possible to modify our algorithms to take into account the Newton polytope, but for simplicity in this paper we only discuss the ``simplex'' case.

For a \emph{smooth} projective hypersurface of degree~$d$, and assuming that $p \neq 2$ and $p \ndivides d$, Lauder obtained the bound $(d^n a p)^{O(1)}$ using the ``deformation method''~\cite{Lau-manyvar}.
The dependence on~$p$ was originally $p^{2+\eps}$, but Lauder conjectured in~\cite{Lau-recursive} that this may be reduced to~$p^{1+\eps}$, and this is consistent with observations of subsequent authors~\cite{Ger-relative,PT-deformation}.
The deformation method has the advantage that for fixed~$p$ the complexity is polynomial in the input size, even for varying dimension.
This is not the case for Theorems~\ref{thm:hs-linear} and~\ref{thm:hs-sqrt}, or for the Lauder--Wan algorithm.

Another approach to the smooth projective case was proposed by Abbott, Kedlaya and Roe~\cite{AKR-picard}.
They did not analyse the complexity of their algorithm; the dependence on~$p$ appears to be $p^{n+\eps}$.
Around 2010, the present author developed variants of the Abbott--Kedlaya--Roe algorithm with complexity~$p^{1+\eps}$ and~$p^{1/2+\eps}$.
These results remain unpublished, although some of the key ideas appear in this paper in modified form.
Ongoing implementation work by Edgar Costa suggests that these variants are feasible in practice.

We emphasise that Theorems~\ref{thm:hs-linear} and~\ref{thm:hs-sqrt} impose no smoothness or sparsity conditions on~$X$.
As far as we are aware, the complexity bound in Theorem~\ref{thm:hs-sqrt} is the best known, as a function of~$p$, for computing the zeta function of an arbitrary (or even smooth) hypersurface over a finite field.
While Theorem~\ref{thm:hs-linear} offers a less favourable time bound, it still matches all previously published algorithms, and has the advantage that the space complexity is only $O(\log p)$.
Apart from the naive algorithm, we do not know of any previous algorithm with space complexity less than exponential in $\log p$.

For the special case of \emph{curves}, the point-counting literature is much richer.
We mention only a few relevant results.

For a curve~$X$ of genus~$g$ over~$\FF_q$, the descendants of Schoof's $\ell$-adic method~\cite{Sch-elliptic,Pil-abelian,AH-counting} compute $Z_X(T)$ in time $(\log q)^{C_g}$ where in general~$C_g$ depends exponentially on~$g$.
If~$X$ is hyperelliptic,~$C_g$ can be taken to be polynomial in~$g$. 
Thus for fixed~$g$, these algorithms have complexity polynomial in $\log p$ and~$a$, but in general the complexity is badly exponential in~$g$.

On the other hand, the various $p$-adic algorithms have complexity polynomial in~$g$ and~$a$, but exponential in $\log p$.
For example, Kedlaya's algorithm for hyperelliptic curves~\cite{Ked-hyperelliptic} has complexity $(g^4 a^3 p)^{1+\eps}$.
Kedlaya's method is very flexible, and the algorithm was subsequently generalised to larger classes of curves.
Until very recently, the most general version available applied to any nondegenerate curve~\cite{CDV-nondegenerate}.
The dependence on~$p$ was not analysed in~\cite{CDV-nondegenerate}, but one expects it to be~$p^{1+\eps}$.
Unfortunately, even in genus~$5$, most curves do not have a nondegenerate model~\cite{CV-nondegeneracy}.
The recent results of Tuitman extend Kedlaya's method to even more general curves~\cite{Tui-general,Tui-P1}, but it is unclear exactly what class of curves is covered.

A complexity bound of the form $(ga)^{O(1)} p^{1/2+\eps}$ first appeared in~\cite{BGS-recurrences}, for hyperelliptic curves, but there the zeta function was determined only modulo~$p$.
The present author subsequently obtained the same bound for the full zeta function~\cite{Har-kedlaya}.
It seemed for some time that the construction of the latter algorithm depended very strongly on the shape of the equation defining the curve, and it was unclear whether it could be generalised much further.
The only positive result in this direction was Minzlaff's result for superelliptic curves~\cite{Min-superelliptic}, which is algebraically very similar to the hyperelliptic case.
Theorem~\ref{thm:hs-sqrt} of this paper shows that in fact the techniques used to obtain the $p^{1/2+\eps}$ bound apply far more generally than previously supposed.

Using Theorem~\ref{thm:hs-linear}, we may sketch a zeta function algorithm for \emph{arbitrary} curves with time complexity $(ga)^{O(1)} p \log^{1+\eps} p$ and space complexity $O((ga)^{O(1)}\log p)$.
The idea is to count points on a \emph{not necessarily smooth} plane model of~$X$.
This gives the right result except possibly at the singularities, and at the points outside the torus.
One then explicitly determines the (finitely many) exceptional points and makes appropriate corrections.
Alternatively, as suggested by the referee, one may take advantage of the Weil Conjectures, and simply remove the factors from the numerator of the zeta function whose roots have the wrong absolute value.
Using~Theorem \ref{thm:hs-sqrt} instead, one expects to obtain time complexity $(ga)^{O(1)} p^{1/2} \log^{2+\eps} p$.
We have not checked the details, or determined the precise exponents of~$a$ and~$g$.

Our strategy for proving Theorems~\ref{thm:hs-linear} and~\ref{thm:hs-sqrt} depends on two key ideas: a new \emph{trace formula} (Theorem~\ref{thm:trace}), and a \emph{deformation recurrence} (Theorem~\ref{thm:deform}), developed in Sections~\ref{sec:trace} and~\ref{sec:recurrences} respectively.

The trace formula expresses $Z_X(T)$ in terms of certain coefficients of powers of an arbitrary $p$-adic lift~$F$ of~$\red F$.
Our framework for developing this formula is strongly influenced by~\cite{LW-counting}.
The algorithms of~\cite{LW-counting} use a different trace formula due to Dwork, evaluating a certain character sum that has been lifted $p$-adically via Dwork's ``splitting function'', whereas our trace formula does not involve any splitting function.

Evaluating the trace formula in the most naive way, by simply expanding out the relevant powers of~$F$ and reading off the appropriate coefficients, leads to a zeta function algorithm whose dependence on~$p$ is $p^{n+\eps}$ (Theorem~\ref{thm:hs-naive}).
This complexity bound is comparable to the main result of~\cite{LW-counting}, but it is achieved by an algorithm that is arguably simpler.

To improve the dependence on~$p$, we must show how to extract coefficients of powers of~$F$ more efficiently than the naive algorithm.
For this, we observe that it is easy to compute coefficients of powers of the auxiliary polynomial $G = x_0^d + \cdots + x_n^d$, since they are just multinomial coefficients.
Theorem~\ref{thm:deform} then gives a recurrence that ``deforms'' powers of~$G$ into powers of~$F$.
This is reminiscent of Lauder's deformation method, but we do not know a precise relationship.
Evaluating the deformation recurrence in the most straightforward way immediately yields a proof of Theorem~\ref{thm:hs-linear}.
Evaluating it instead with an algorithm of Bostan, Gaudry and Schost~\cite{BGS-recurrences} leads to Theorem~\ref{thm:hs-sqrt}.

While Theorem~\ref{thm:deform} looks quite innocuous, and its proof is very simple, the recurrence has the crucial property that it imposes no smoothness hypotheses on~$F$.
From an algebraic point of view, it avoids the denominators that typically appear in zeta function algorithms based on $p$-adic cohomology, i.e., arising from divisions by ``resultants'' or ``discriminants''.
Ultimately, this is why the new algorithms are applicable to a completely general hypersurface, and hence any variety whatsoever.

\subsection{Hypersurfaces over $\ZZ$}
Let $\PP^n_\ZZ = \Proj \ZZ[x_0, \ldots, x_n]$ denote projective $n$-space over~$\ZZ$, and let~$\TT^n_\ZZ$ be the open subscheme obtained as the complement of the zero locus of $x_0 \cdots x_n$.
Let $F \in \ZZ[x_0, \ldots, x_n]$ be nonzero and homogeneous of degree $d \geq 1$, and let~$X$ be the closed subscheme of~$\TT^n_\ZZ$ defined by~$F$.
For any prime~$p$, let~$\red F_p$ be the image of~$F$ in $\FF_p[x_0, \ldots, x_n]$, and let~$X_p$ be the hypersurface in $\TT^n_{\FF_p}$ defined by~$\red F_p$.
Note that if~$p$ divides all the coefficients of~$F$, then $\red F_p = 0$ and $X_p = \TT^n_{\FF_p}$.

Let $N \geq 2$, and consider the problem of computing $Z_{X_p}(T)$ for all $p < N$.
Using the Lauder--Wan algorithm for each prime separately leads to the complexity bound $N^{O(n)}$.
Theorem~\ref{thm:hs-linear} improves this to $N^{2+\eps}$, and Theorem~\ref{thm:hs-sqrt} reduces it further to $N^{3/2+\eps}$.
The next result achieves the bound~$N^{1+\eps}$, which is optimal up to logarithmic factors, by treating all primes simultaneously.
This generalises the author's result for hyperelliptic curves~\cite{Har-avgpoly}.
We denote by~$\norm F$ the maximum of the absolute values of the coefficients of~$F$.
\begin{thm}
\label{thm:hs-avgpoly}
There exists an explicit deterministic algorithm with the following properties.
The input consists of positive integers~$N$, $n$, $d$, and a homogeneous polynomial $F \in \ZZ[x_0, \ldots, x_n]$ of degree~$d$, defining a hypersurface~$X$ in~$\TT^n_\ZZ$ as above.
The output is the sequence of zeta functions $Z_{X_p}(T)$ for all primes $p < N$, $p \ndivides d$.
The algorithm has time complexity
 \[ 2^{8n^2 + 16n} n^{4n+6+\eps} (d+1)^{4n^2+7n+\eps} N \log^2 N \log^{1+\eps}(N \norm F) \]
and space complexity
 \[ O(2^{4n^2+11n} n^{2n+4} (d+1)^{2n^2+5n} N \log N \log(ndN \norm F)). \]
\end{thm}

The main idea of the proof is to apply the machinery of the ``accumulating remainder tree'' (ART) to the recurrence mentioned earlier.
The ART was introduced in~\cite{CGH-wilson} for the purpose of computing the Wilson quotients $(p-1)! \pmod{p^2}$ for many~$p$ simultaneously, and its subsequent generalisation to matrices played a central role in~\cite{Har-avgpoly}.

However, there is an important difference between~\cite{Har-avgpoly} and the present paper.
In~\cite{Har-avgpoly}, the ART was coupled with the technique of ``reduction towards zero''.
The latter involved an algebraic rearrangement of the problem to ensure that the matrices defining the recurrence were independent of~$p$.
In the present situation we have been unable to make the recurrence matrices independent of~$p$.
Instead, we propose the following workaround: we replace~$p$ by a formal variable~$k$, and run the ART algorithm over a truncated power series ring in~$k$.
Thus~$k$ plays the role of a ``generic prime'' that has not yet been specialised to an actual prime number.
At the very end, we specialise to $k = p$, separately for each prime~$p$.

(Incidentally, this shows that the difficulties that led to the introduction of ``reduction towards zero'' in the first place were to some extent a red herring.
That is, for hyperelliptic curves, one could obtain results similar to~\cite{Har-avgpoly} by combining the original reduction formulae from~\cite{Har-kedlaya} with the ``generic prime'' technique, and avoid ``reduction towards zero'' altogether.
It seems likely that this would improve the exponent of~$g$, the genus of the curve, in the complexity bound.
We have not yet checked the details, but we observe that for the simpler problem of computing Hasse--Witt matrices of hyperelliptic curves, analogous considerations explain much of the improvement in performance between~\cite{HS-hassewitt} and~\cite{HS-hassewitt2}.)

As above, for an arbitrary curve of genus~$g$ over~$\QQ$, we may apply Theorem~\ref{thm:hs-avgpoly} to a (possibly singular) plane model, and then correct for the exceptional points.
We thus expect to be able to compute the Euler factors of its zeta function for $p < N$ in time $g^{O(1)} N \log^{3+\eps} N$ (ignoring the dependence on the size of the coefficients of the polynomial defining the plane model).

\subsection{Arithmetic schemes}
We conclude this section by showing how to deduce Theorem~\ref{thm:scheme} from the hypersurface case.
\begin{proof}[of Theorem~\ref{thm:scheme}]
Let~$X$ be a scheme of finite type over~$\ZZ$.
We may assume that~$X$ is reduced, as the definition of $\zeta_X(s)$ depends only on the closed points of~$X$.
We will first show that $\zeta_X(s)$ may be expressed as a finite product $\prod_i \zeta_{X_i}(s)^{e_i}$, where each $e_i = \pm 1$, and where each~$X_i$ is a hypersurface in $\TT^{n_i}_\ZZ$ for some~$n_i$, in the sense of Theorem~\ref{thm:hs-avgpoly}.

For this, we will repeatedly use the fact that if~$X$ is a disjoint union $U \cup Y$, where~$U$ is an open subscheme and~$Y$ is a closed subscheme (with the reduced closed subscheme structure), then both~$U$ and~$Y$ are of finite type over~$\ZZ$, and $\zeta_X(s) = \zeta_U(s) \zeta_Y(s)$.

Since~$X$ is of finite type over~$\ZZ$, it has a finite cover by open affines, say $X = U_1 \cup \cdots \cup U_n$, where each~$U_i$ is the spectrum of a finitely generated $\ZZ$-algebra.
Then~$X$ is the disjoint union $U_1 \cup X'$ where $X' = (U_2 \cup \cdots \cup U_n) \setminus U_1 = (U_2 \setminus U_1) \cup \cdots \cup (U_n \setminus U_1)$.
The latter is a cover of~$X'$ by $n-1$ spectra of finitely generated $\ZZ$-algebras.
Applying the procedure recursively to~$X'$, we obtain a representation of~$X$ as a disjoint union $V_1 \cup \cdots \cup V_n$ where each~$V_i$ is the spectrum of a finitely generated $\ZZ$-algebra.
Thus $\zeta_X(s) = \prod_i \zeta_{V_i}(s)$, and we have reduced the problem to the case that~$X$ is the spectrum of a finitely generated $\ZZ$-algebra.

Thus let $X = \Spec \ZZ[x_1, \ldots, x_m]/(F_1, \ldots, F_k)$.
We apply the ``inclusion--exclusion trick'' of~\cite[\S3]{Wan-zeta}.
For each nonempty subset~$S \subseteq \{1, \ldots, k\}$, let $F_S = \prod_{i \in S} F_i$ and $X_S = \Spec \ZZ[x_1, \ldots, x_m]/(F_S)$.
Then $\zeta_X(s) = \prod_S \zeta_{X_S}(s)^{(-1)^{1+|S|}}$, so we have reduced to the case $X = \Spec \ZZ[x_1, \ldots, x_m]/(F)$, i.e., an affine hypersurface.

Finally, for such~$X$, for each subset $T \subseteq \{1, \ldots, m\}$, let~$F_T$ be the polynomial obtained from~$F$ by substituting $x_i = 0$ for each $i \notin T$, and let $X_T$ be the hypersurface defined by~$F_T$ in the affine torus $\Spec \ZZ[x_i : i \in T][x_0]/(1 - x_0 \prod_{i \in T} x_i)$.
Then $\zeta_X(s) = \prod_T \zeta_{X_T}(s)$, and we obtain the desired product representation.

Now return to the general case of a scheme~$X$ of finite type over~$\ZZ$, and let $\zeta_X(s) = \prod_i \zeta_{X_i}(s)^{e_i}$ be a suitable product representation as above.
The proof of the decomposition shows that it is compatible with the Euler product, i.e., for each~$p$ we have $\zeta_{X_p}(s) = \prod_i \zeta_{(X_i)_p}(s)^{e_i}$.
Theorems~\ref{thm:hs-linear},~\ref{thm:hs-sqrt} and~\ref{thm:hs-avgpoly} imply that parts (a), (b) and (c) of Theorem~\ref{thm:scheme} hold for each~$X_i$.
Also, for each~$i$, there exist integers~$d_i$ and~$c_i$ (not depending on~$p$) such that the numerator and denominator of $Z_{(X_i)_p}(T)$ are polynomials of degree at most~$d_i$, with coefficients bounded by~$p^{c_i}$.
Thus $Z_{(X_i)_p}(T)$ occupies space $O(\log p)$, and given $Z_{(X_i)_p}(T)$ for all~$i$, we may compute $Z_{X_p}(T) = \prod_i Z_{(X_i)_p}(T)^{e_i}$ in time $\log^{1+\eps} p$ and space $O(\log p)$.
(See Section~\ref{sec:complexity} for generalities on fast polynomial arithmetic.)
\end{proof}

One may also prove results similar to Theorem~\ref{thm:scheme}(a) and (b) for a scheme~$X$ of finite type over~$\FF_q$, $q = p^a$.
One approach is to apply Theorem~\ref{thm:scheme} to the scheme~$X'$ over~$\ZZ$ obtained by composing the morphism $X \to \Spec \FF_q$ with the morphism $\Spec \FF_q \to \Spec \ZZ$.
Note that this computes $Z_{(X')_p}(T) = Z_X(T^a)$ rather than $Z_X(T)$ directly.
A more efficient method is to use the same decomposition strategy as in the proof of Theorem~\ref{thm:scheme} to reduce to the case of hypersurfaces over~$\FF_q$, and then to apply Theorems~\ref{thm:hs-linear} and~\ref{thm:hs-sqrt} to each hypersurface.

\section{Basic complexity results}
\label{sec:complexity}

In this section we recall some basic complexity results that will be used freely throughout the paper.

Adding or subtracting $n$-bit integers may be achieved in time~$O(n)$.
Multiplication of $n$-bit integers, and division with remainder of $n$-bit integers, have time complexity $n \log^{1+\eps} n$ and space complexity $O(n)$ using fast Fourier transform and Newton iteration methods~\cite[Ch.~8, 9]{vzGG-compalg}.

Let $d \geq 1$ and $n \geq 1$. If $g, h \in \ZZ[t]$ have degree at most~$d$, and if the coefficients of~$g$, $h$ and~$gh$ have at most~$n$ bits, then~$gh$ may be computed in time $dn \log^{1+\eps}(dn)$ and space $O(dn)$ by Kronecker substitution~\cite[Ch.~8]{vzGG-compalg}.

Let~$p$ be a prime and let $q = p^a$, $a \geq 1$.
We represent the field~$\FF_q$ as $\FF_p[t]/\red f$, where $\red f \in \FF_p[t]$ is monic and irreducible, of degree~$a$.
We will always assume that~$\red f$ is given as input.
Elements of~$\FF_q$ are represented by polynomials $g \in \FF_p[t]$ of degree less than~$a$, and thus occupy space $O(a \log p)$.
Addition and subtraction in~$\FF_q$ require time $O(a \log p)$.
Multiplication in~$\FF_q$ has time and space complexity $(a \log p)^{1+\eps}$ and $O(a \log p)$, using the Cantor--Kaltofen algorithm~\cite[Ch.~8]{vzGG-compalg}.
(An alternative is Kronecker substitution, but this leads to suboptimal complexity bounds if~$a$ is very large relative to~$p$.)
Division in~$\FF_q$ has time and space complexity $(a \log p)^{1+\eps}$ and $O(a \log p)$, using the fast extended Euclidean algorithm~\cite[Ch.~11]{vzGG-compalg}.

We denote by~$\ZZ_q$ the ring of Witt vectors over~$\FF_q$, i.e., the ring of integers of the unique unramified extension of~$\QQ_p$ of degree~$a$, so that $\ZZ_q/p\ZZ_q \cong \FF_q$.
We will need to perform arithmetic in finite precision approximations $\ZZ_q/p^\lambda \ZZ_q$ for $\lambda \geq 1$.
To represent this ring, we choose an arbitrary lift $f \in (\ZZ/p^\lambda \ZZ)[t]$ of~$\red f$, monic of degree~$a$, so that $\ZZ_q/p^\lambda \ZZ_q \cong (\ZZ/p^\lambda\ZZ)[t]/f$.
Thus elements of $\ZZ_q/p^\lambda\ZZ_q$ are represented by polynomials $g \in (\ZZ/p^\lambda\ZZ)[t]$ of degree less than~$a$, and these occupy space $O(\lambda a \log p)$.
As above, addition and subtraction in $\ZZ_q/p^\lambda\ZZ_q$ have complexity $O(\lambda a \log p)$, and multiplication and division require time $\lambda \log^{1+\eps}(2\lambda) (a \log p)^{1+\eps}$ and space $O(\lambda a \log p)$.
(The cruder time bound $(\lambda a \log p)^{1+\eps}$ is not quite strong enough to prove the main results in the form we have stated them.)

Let $\phi : \FF_q \to \FF_q$ be the absolute Frobenius map $u \mapsto u^p$, so that~$\phi^a$ is the identity on~$\FF_q$.
For $0 \leq j < a$ we may compute $\phi^j(u) = u^{p^j}$ via ``binary powering'' using $O(\log(p^j)) = O(a \log p)$ multiplications in~$\FF_q$, i.e., in time $(a^2 \log^2 p)^{1+\eps}$ and space $O(a \log p)$~\cite[Ch.~4]{vzGG-compalg}.

We use the same notation~$\phi$ for the corresponding Frobenius map on~$\ZZ_q$, i.e., the unique automorphism of~$\ZZ_q$ that lifts $\phi : \FF_q \to \FF_q$.
To compute $\phi^j(u)$ for $u \in \ZZ_q/p^\lambda\ZZ_q$, for simplicity we use the following algorithm suggested in~\cite{Ked-hyperelliptic}.
First compute $\phi^j(t) \pmod p$ in~$\FF_q$, and use Newton's method to lift the result to a root~$\alpha$ of~$f$ in $\ZZ_q/p^\lambda\ZZ_q$.
Then evaluate~$g$ at~$\alpha$, where $g \in (\ZZ/p^\lambda\ZZ)[t]$ is the polynomial representing~$u$.
This costs time $\lambda \log^{1+\eps}(2\lambda) (a^2 \log^2 p)^{1+\eps}$ and space $O(\lambda a \log p)$.
To simplify matters later, we have included here the cost of all ``precomputations''.
Superior bounds may be achieved by more elaborate algorithms, such as that of~\cite{Hub-unramified}.

For $N \geq 2$, the set of primes $p < N$ may be enumerated in time $N \log^{2+\eps} N$ and space~$O(N)$.
For example, apply~\cite[Proposition 2.2]{CGH-wilson} to a sequence of intervals of width $O(N / \log N \log \log N)$.

Suppose that~$A$ is an $m \times n$ array of objects of bit size~$\ell$.
In the Turing model, we may transpose the array, i.e., switch from ``row-major'' to ``column-major'' order, in time $O(mn\ell \log \min(m, n))$ and space $O(m n \ell)$~\cite[Lemma 18]{BGS-recurrences}.
Such transposition steps will occur frequently in our algorithms; for example, in the proof of Lemma~\ref{lem:bd-zeta1}, we must transpose after the remainder tree, and again just before the fast Chinese remaindering.
The cost of the transpositions will always be dominated by the cost of actual arithmetic, and we will not mention it again.

The remaining results are stated in terms of an algebraic complexity model over a ring~$R$, where ``time'' means the number of ring operations in~$R$, and ``space'' counts the number of elements of~$R$ that must be stored at any point during a computation.
We leave it to the reader to formulate the corresponding statements in the Turing model, for the specific rings that arise.

For computing the product of two $m \times m$ matrices over a ring~$R$, for simplicity we will use the classical algorithm.
This costs $O(m^3)$ ring operations, and uses space $O(m^2)$.

For multivariate polynomials, say in $R[x_0, \ldots, x_n]$, we always assume that the dense representation is used.
For example, if $H \in R[x_0, \ldots, x_n]$ is homogeneous of degree~$d$, we will assume that the monomials of degree~$d$ are ordered (say) lexicographically, and that the coefficients are presented in a linear array corresponding to this ordering.

If $G, H \in R[x_0, \ldots, x_n]$ are homogeneous of degree at most~$d$, their product may be computed using multivariate Kronecker substitution, i.e.,~$GH$ may be recovered from the univariate (non-homogeneous) product
 \[ G(1, t, t^{2d}, \ldots, t^{(2d)^{n-1}}) H(1, t, t^{2d}, \ldots, t^{(2d)^{n-1}}) \in R[t]. \]
By the Cantor--Kaltofen theorem, this may be achieved using $(2d)^n \log^{1+\eps}((2d)^n)$ ring operations in~$R$.

\section{The trace formula}
\label{sec:trace}

\subsection{The trace formula}

For a domain~$R$ and an integer $n \geq 1$, we denote the multivariate polynomial ring $R[x_0, \ldots, x_n]$ by simply~$R[x]$.
If $u = (u_0, \ldots, u_n) \in \ZZ^{n+1}$, we write $\deg u = u_0 + \cdots + u_n$, and denote by~$F_u$ the coefficient of $x^u = x_0^{u_0} \cdots x_n^{u_n}$ in~$F$, with the understanding that $F_u = 0$ if any component~$u_i$ is negative.
For $k \geq 0$ we denote by $R[x]_k$ the submodule of homogeneous polynomials of degree~$k$.
It is spanned by the monomials~$x^u$ for $u \in B_k = \{u \in \NN^{n+1}: \deg u = k\}$.
Its rank over~$R$ is $|B_k| = \binom{n + k}n$.

Now let~$p$ be a prime, $a \geq 1$ and $q = p^a$.
Define maps $\phi, \psi : \ZZ_q[x] \to \ZZ_q[x]$ by
 \[ \phi(G) = \sum_u \phi(G_u) x^{pu}, \qquad \psi(G) = \sum_u \phi^{-1}(G_{pu}) x^u. \]
Then~$\psi$ is a left inverse of~$\phi$, i.e., $\psi \circ \phi$ is the identity on $\ZZ_q[x]$.
In general~$\phi$ and~$\psi$ are not $\ZZ_q$-linear, but they are respectively $\phi$-semilinear and $\phi^{-1}$-semilinear, i.e.~$\phi(bG) = \phi(b)\phi(G)$ and $\psi(bG) = \phi^{-1}(b)\psi(G)$ for $b \in \ZZ_q$.
In particular, since~$\phi^a$ is the identity on~$\ZZ_q$, we have
\begin{equation}
\label{eq:psi-power}
 \psi^a(G) = \sum_u G_{qu} x^u.
\end{equation}

For $k \geq 1$ and $H \in \ZZ_q[x]_k$, let $T_H : \ZZ_q[x] \to \ZZ_q[x]$ be the multiplication operator $G \mapsto HG$, and let
 \[ A_H = \psi \circ T_{H^{p-1}}. \]
Note that~$A_H$ is $\phi^{-1}$-semilinear, and maps $\ZZ_q[x]_k$ into $\ZZ_q[x]_k$, because if $\deg G = k$, then $\deg \psi(H^{p-1} G) = (k(p-1) + k) / p = k$.
For $m \geq 1$, we also define
 \[ H^{(m)} = (H \cdot \phi(H) \cdots \phi^{m-1}(H))^{p-1}. \]
It follows immediately that
\begin{equation}
\label{eq:A-power}
 A_H^m = \psi^m \circ T_{H^{(m)}}
\end{equation}
for any $m \geq 1$.

\begin{thm}[(Trace formula)]
\label{thm:trace}
Let $\red F \in \FF_q[x]_d$ and let~$X$ be the hypersurface in $\TT^n_{\FF_q}$ cut out by~$\red F$.
Let~$r$, $\lambda$ and~$\tau$ be positive integers satisfying
\begin{equation} 
\label{eq:assumption}
 \tau \geq \frac{\lambda}{(p-1)ar}.
\end{equation}
Let $F \in \ZZ_q[x]_d$ be any lift of~$\red F$.
Then
 \[ |X(\FF_{q^r})| = (q^r - 1)^n \sum_{s=0}^{\lambda + \tau - 1} \alpha_s \trace (A_{F^s}^{ar}) \pmod{p^\lambda}, \]
where
 \[ \alpha_s = (-1)^s \sum_{t=0}^{\tau - 1} \binom{-\lambda}t \binom{\lambda}{s-t} \in \ZZ, \]
and where $A_{F^s}$ is regarded as a linear operator on $\ZZ_q[x]_{ds}$.
\end{thm}

Before giving the proof, we comment briefly on hypothesis~\eqref{eq:assumption}.
If $p \geq 1 + \frac{\lambda}{ar}$, then~\eqref{eq:assumption} is satisfied for $\tau = 1$, and the trace formula becomes simply
 \[ |X(\FF_{q^r})| = (q^r - 1)^n \sum_{s=0}^\lambda (-1)^s \binom{\lambda}s \trace(A_{F^s}^{ar}) \pmod{p^\lambda}. \]
For smaller~$p$ one may need to take $\tau > 1$, and then the trace formula involves higher powers of~$F$.
Thus~$\tau$ should be thought of as a ``fudge factor'' that corrects the trace formula for small~$p$.

\begin{proof*}
Let $H = F^{(ar)}$ and let
 \[ J = \sum_{s=0}^{\lambda+\tau-1} \alpha_s H^s = (1 - H)^\lambda \sum_{t=0}^{\tau - 1} (-1)^t \binom{-\lambda}t H^t. \]
Note that~$H$ is homogeneous, but~$J$ is not.
We claim that for any $c \in (\ZZ_{q^r}^*)^{n+1}$,
\begin{equation}
\label{eq:indicator}
 J(c) = \begin{cases} 1 \pmod{p^\lambda} & \text{if $\red F(\red c) = 0$}, \\
                      0 \pmod{p^\lambda} & \text{otherwise}, \end{cases}
\end{equation}
where~$\red c$ is the image of~$c$ in $(\FF_{q^r}^*)^{n+1}$.
Indeed, if $\red F(\red c) \neq 0$, then 
 \[ H(c) = (\red F(\red c) \red F(\red c)^p \cdots \red F(\red c)^{p^{ar-1}})^{p-1} = \red F(\red c)^{q^r - 1} = 1 \pmod p. \]
Thus $(1 - H(c))^\lambda = 0 \pmod{p^\lambda}$, so $J(c) = 0 \pmod{p^\lambda}$.
On the other hand, suppose that $\red F(\red c) = 0$.
Then $H(c) = 0 \pmod{p^{(p-1)ar}}$, so by~\eqref{eq:assumption} we have $H(c)^t = 0 \pmod{p^\lambda}$ for all $t \geq \tau$.
In particular
 \[ \sum_{t=0}^{\tau-1} (-1)^t \binom{-\lambda}t H(c)^t = \sum_{t=0}^{\infty} (-1)^t \binom{-\lambda}t H(c)^t = (1 - H(c))^{-\lambda} \pmod{p^\lambda}, \]
so $J(c) = 1 \pmod{p^\lambda}$.

Now let~$\Sigma$ be the set of $(q^r-1)$-th roots of unity in $\ZZ_{q^r}$, i.e., the set of Teichm\"uller lifts of elements of $\FF_{q^r}^*$.
For any $w \in \NN^{n+1}$ we have
 \[ \sum_{c \in \Sigma^{n+1}} c^w = \sum_{c_0 \in \Sigma} c_0^{w_0} \cdots \sum_{c_n \in \Sigma} c_n^{w_n} = \begin{cases} (q^r - 1)^{n+1} & \text{if $q^r - 1 \divides w_i$ for all $0 \leq i \leq n$}, \\ 0 & \text{otherwise}, \end{cases} \]
so summing~\eqref{eq:indicator} over $c \in \Sigma^{n+1}$ we obtain
 \[ |X(\FF_{q^r})| = (q^r - 1)^n \sum_{w \in \NN^{n+1}} J_{(q^r - 1)w}  \pmod{p^\lambda}. \]
Since $\deg H^s = d(1 + p + \cdots + p^{ar-1})(p-1)s = ds(q^r - 1)$, this becomes
 \[ |X(\FF_{q^r})| = (q^r - 1)^n \sum_{s=0}^{\lambda+\tau-1} \alpha_s \sum_{w \in B_{ds}} (H^s)_{(q^r - 1)w} \pmod{p^\lambda}. \]
By~\eqref{eq:psi-power} we have
 \[ \sum_{w \in B_{ds}} (H^s)_{(q^r - 1)w} = \sum_{w \in B_{ds}} (H^s x^w)_{q^r w} = \sum_{w \in B_{ds}} (\psi^{ar}(H^s x^w))_w, \]
and then since $H^s = (F^s)^{(ar)}$, equation~\eqref{eq:A-power} shows that this is equal to
\multbox
\begin{eqnarray*}
 \sum_{w \in B_{ds}} (A_{F^s}^{ar} x^w)_w = \trace(A_{F^s}^{ar}).
\end{eqnarray*}
\emultbox
\end{proof*}

We may give a more computationally explicit description of $A_{F^s}^{ar}$ as follows.
\begin{lem}
\label{lem:M-prod}
Let $F \in \ZZ_q[x]_d$.
The matrix of $A_{F^s}^a$ on $\ZZ_q[x]_{ds}$, with respect to the basis~$B_{ds}$, is given by
 \[ \phi^{a-1}(M_s) \cdots \phi(M_s) M_s, \]
where~$M_s$ is the square matrix defined by
 \[ (M_s)_{v,u} = (F^{(p-1)s})_{pv - u} \]
for $u, v \in B_{ds}$, and where~$\phi$ acts componentwise on matrices.
\end{lem}
\begin{proof*}
For $G \in \ZZ_q[x]_{ds}$, write~$[G]$ for the coordinate vector of~$G$ with respect to~$B_{ds}$.
Observe that
 \[ (A_{F^s} x^u)_v = (\psi(F^{(p-1)s} x^u))_v = \phi^{-1}((F^{(p-1)s} x^u)_{pv}) = \phi^{-1}((M_s)_{v,u}), \]
and so, since~$A_{F^s}$ is $\phi^{-1}$-semilinear,
 \[ [A_{F^s} G] = \phi^{-1}(M_s) \phi^{-1}([G]). \]
Iterating~$a$ times, and using the fact that~$\phi^a$ is the identity on~$\ZZ_q$, we find that
\multbox
\begin{eqnarray*}
 [A_{F^s}^a G] = \phi^{a-1}(M_s) \cdots \phi(M_s) M_s [G].
\end{eqnarray*}
\emultbox
\end{proof*}

\subsection{Complexity of evaluating the trace formula}

Next we carry out a straightforward estimate of the complexity of evaluating the trace formula to determine $Z_X(T)$, assuming that the~$M_s$ are known.
In subsequent sections we will study efficient algorithms for computing the~$M_s$ themselves.

\begin{lem}
\label{lem:bd-zeta2}
Let $\lambda = 2na(4d+4)^n$.
Given as input $M_s \pmod{p^\lambda}$ for $1 \leq s \leq 2\lambda$, we may compute $Z_X(T)$ in time
 \[ 2^{6n^2 + 13n} n^{3n+3+\eps} (d+1)^{3n^2+6n+\eps} a^{3n+4+\eps} \log^{2+\eps} p \]
and space
 \[ O(2^{4n^2 + 9n} n^{2n+2} (d+1)^{2n^2+4n} a^{2n+3} \log p). \]
\end{lem}
For the proof we need two preliminary results.
\begin{lem}
\label{lem:bd-trace}
Let $\lambda \geq 1$, $1 \leq s \leq 2\lambda$ and $D \geq 1$.
Given as input $M_s \pmod{p^\lambda}$, we may compute $\trace(A_{F^s}^{ar}) \pmod{p^\lambda}$ for all $1 \leq r \leq 2D$ in time
 \[ D (n+2d\lambda)^{3n} \lambda \log^{1+\eps}(2\lambda) a^{2+\eps} \log^{2+\eps} p \]
and space
 \[ O(( D + (n + 2d\lambda)^{2n}) \lambda a \log p). \]
\end{lem}
\begin{proof}
First compute $\phi^{a-1}(M_s) \cdots \phi(M_s) M_s \pmod{p^\lambda}$, using a modified binary powering algorithm (see~\cite[\S5]{Ked-hyperelliptic} or~\cite[Lemma 32]{LW-counting}).
This requires $O(\log a)$ matrix multiplications and $O(\log a)$ ``matrix Frobenius'' operations, i.e., applying~$\phi^j$, for some $0 \leq j < a$, to each entry of a matrix.
The matrix size is
\begin{equation}
\label{eq:B-bound}
 |B_{ds}| = \binom{n + ds}{n} \leq (n + ds)^n \leq (n + 2d\lambda)^n.
\end{equation}
Thus the time complexity is
\begin{multline*}
 (\log a) \lambda \log^{1+\eps} (2\lambda) \big( |B_{ds}|^3 (a \log p)^{1+\eps} + |B_{ds}|^2 (a^2 \log^2 p)^{1+\eps} \big) \\
  = (n+2d\lambda)^{3n} \lambda \log^{1+\eps}(2\lambda) a^{2+\eps} \log^{2+\eps} p.
\end{multline*}
The space complexity is just the size of the matrix, namely
 \[ O(|B_{ds}|^2 \lambda a \log p) = O((n + 2d\lambda)^{2n} \lambda a \log p). \]
By Lemma~\ref{lem:M-prod} this yields the matrix of $A_{F^s}^a \pmod{p^\lambda}$.
Then compute~$D$ successive powers, each one requiring a single matrix multiplication.
The matrix may be overwritten as we proceed.
The output, i.e., the sequence of traces, occupies space $O(D \lambda a \log p)$.
\end{proof}

\begin{lem}
\label{lem:bd-zeta1}
Given as input $|X(\FF_{q^r})|$ for $1 \leq r \leq 2(4d+4)^n$, we may compute $Z_X(T)$ in time
 \[ 2^{6n} n^{3+\eps} (d+1)^{3n+\eps} a^{1+\eps} \log^{1+\eps} p \]
and space
 \[ O(2^{6n} n (d+1)^{3n} a \log p). \]
\end{lem}
\begin{proof*}
Our argument follows closely the ideas of~\cite[\S6.4]{LW-counting}.
Write $Z_X(T) = G(T) / H(T) = \sum_{k \geq 0} c_k T^k$ where $G, H \in 1 + T\ZZ[T]$ are relatively prime.
According to~\cite[Theorem~1A]{Bom-exponential} we have $\deg G \leq D$ and $\deg H \leq D$ where $D = (4d+4)^n$.
By hypothesis we are given as input $|X(\FF_{q^r})|$ for $1 \leq r \leq 2D$.

We may bound the maximum bit size~$B$ of the coefficients of $G(T)$ and $H(T)$ as follows.
Write $G(T) = \prod_i (1 - \alpha_i T)$ and $H(T) = \prod_i (1 - \beta_i T)$ for $\alpha_i, \beta_i \in \CC$.
From the trivial bound $|X(\FF_{q^r})| \leq q^{nr}$ we see that $\sum_{r \geq 1} |X(\FF_{q^r})| T^r/r$ converges for $|T| < q^{-n}$, so $\beta_i \leq q^n$ and $\alpha_i \leq q^n$.
Thus the coefficients of~$G(T)$ and~$H(T)$ are bounded in absolute value by $2^D (q^n)^D$, and we obtain $B = O(n D a \log p)$.

The sequence~$\{c_k\}$ is linearly recurrent with characteristic polynomial~$H(T)$, i.e., each~$c_k$ for $k > D$ is a suitable linear combination of the previous~$D$ terms.
Let $\delta \in \ZZ$ be the (unknown) resultant of~$G$ and~$H$.
It is nonzero because~$G$ and~$H$ are relatively prime, and by the Hadamard bound we have $|\delta| \leq (2D)^D (2^B)^{2D}$, so $\log|\delta| = O(n D^2 a \log p)$.
For any prime~$\ell$, the sequence $\{c_k \pmod{\ell}\}$ is linearly recurrent over~$\ZZ/\ell \ZZ$.
If $\ell \ndivides \delta$, then~$G$ and~$H$ are relatively prime in $(\ZZ/\ell\ZZ)[T]$, and the corresponding characteristic polynomial is exactly $H(T) \pmod \ell$; otherwise it is of strictly smaller degree.
Our strategy will be to compute $H(T) \pmod \ell$ for sufficiently many ``small'' primes~$\ell$ and then reconstruct~$H(T)$ via the Chinese remainder theorem.

Accordingly, let $B' = O(n D^2 a \log p)$ be a bound for the bit size of~$|\delta|$, and take a collection~$\LL$ of primes $\ell > 2D$ such that $\prod_{\ell \in \LL} \ell \geq 2^{B' + B + 1}$.
By the prime number theorem we may assume that $|\LL| = O(B'/\log B')$ and that $\ell = O(B')$ for each~$\ell$.
Using a fast remainder tree~\cite[\S10.1]{vzGG-compalg}, we may compute $|X(\FF_{q^r})| \pmod{\ell}$ for all $1 \leq r \leq 2D$ and all $\ell \in \LL$ in time $D B' \log^{2+\eps} B'$ and space $O(DB')$.
Now, for each~$\ell$, use a fast series exponential algorithm~\cite[\S9]{Ber-fastmult} to compute $c_k \pmod{\ell}$ for $0 \leq k \leq 2D$ in time $D \log^{1+\eps} D \log^{1+\eps} \ell$ and space $O(D \log \ell)$ (the condition $\ell > 2D$ ensures $\ell$-integrality), and then a fast variant of the Berlekamp--Massey algorithm~\cite[\S11]{vzGG-compalg} to find the characteristic polynomial of $\{c_k \pmod{\ell}\}$ in time $D \log^{2+\eps} D \log^{1+\eps} \ell$ and space $O(D \log \ell)$.
Over all~$\ell$ the time cost is $(B' / \log B') D \log^{2+\eps} D \log^{1+\eps} B' = B' D \log^{2+\eps} D \log^{\eps} B'$.
The product of those~$\ell$ such that $\ell \divides \delta$ has at most~$B'$ bits, and these~$\ell$ may be recognised as those for which the characteristic polynomial modulo~$\ell$ does not have maximal degree.
The rest of the primes have product at least~$2^{B+1}$.
After normalising the characteristic polynomials modulo these ``good'' primes so that their constant term is~$1$, we may combine them using fast Chinese remaindering~\cite[\S10.3]{vzGG-compalg} to obtain $H(T)$ in time $D B \log^{2+\eps} B$ and space $O(DB)$.
Finally we obtain~$G(T)$ by multiplying~$H(T)$ by $Z_X(T)$ (or indeed by repeating the whole algorithm for $1/Z_X(T)$).

The total time complexity is
\begin{multline*}
DB' (\log^{2+\eps} B' + \log^{2+\eps} D \log^\eps B')  = DB' \log^{2+\eps} B' \\
 \begin{split}
 & = n D^3 a \log p \log^{1+\eps}(n D^2 a \log p) \\
 & = n^{1+\eps} (4d+4)^{3n} a^{1+\eps} \log^{1+\eps} p \log^{2+\eps}((4d+4)^{2n}) \\
 & = 2^{6n} n^{3+\eps} (d+1)^{3n+\eps} a^{1+\eps} \log^{1+\eps} p,
 \end{split}
\end{multline*}
and the space complexity is
\multbox
\begin{eqnarray*}
 O(DB') = O(nD^3 a \log p) = O(2^{6n} n (d+1)^{3n} a \log p).
\end{eqnarray*}
\emultbox
\end{proof*}

\begin{proof}[of Lemma~\ref{lem:bd-zeta2}]
To apply Lemma~\ref{lem:bd-zeta1} we must first compute $|X(\FF_{q^r})|$ for $1 \leq r \leq 2D$, where $D = (4d + 4)^n$.
For such~$r$ we have the trivial bound $|X(\FF_{q^r})| < q^{nr} \leq q^{2nD}$.
Thus it suffices to compute $|X(\FF_{q^r})| \pmod{p^\lambda}$ for $\lambda = 2naD$.
For this we apply Theorem~\ref{thm:trace}, taking $\tau = \lceil \lambda/(p-1)ar \rceil$ so that~\eqref{eq:assumption} is satisfied.
Of course $\tau \leq \lambda$, so it suffices to compute $\trace(A_{F^s}^{ar}) \pmod{p^\lambda}$ for $1 \leq s \leq 2\lambda$.
Lemma~\ref{lem:bd-trace} achieves this in time
\begin{multline*}
  \lambda D (n+2d\lambda)^{3n} \lambda \log^{1+\eps}(2\lambda) a^{2+\eps} \log^{2+\eps} p \\
  \begin{split}
    & = n^2 D^3 (n+4ndaD)^{3n} \log^{1+\eps}(4naD) a^{4+\eps} \log^{2+\eps} p \\
    & = n^{3n+2+\eps} (4d+4)^{3n} (5da(4d+4)^n)^{3n} \log^{1+\eps}((4d+4)^n) a^{4+\eps} \log^{2+\eps} p \\
    & = 2^{6n^2 + 13n} n^{3n+3+\eps} (d+1)^{3n^2+6n+\eps} a^{3n+4+\eps} \log^{2+\eps} p
  \end{split}
\end{multline*}
and space
\begin{align*}
   O(\lambda (D + (n + 2d\lambda)^{2n})\lambda a \log p)
     & = O(n^2 (n + 4nda(4d+4)^n)^{2n} (4d+4)^{2n} a^3 \log p) \\
     & = O(n^{2n+2} (5da(4d+4)^n)^{2n} (4d+4)^{2n} a^3 \log p) \\
     & = O(2^{4n^2+9n} n^{2n+2} (d+1)^{2n^2+4n} a^{2n+3} \log p).
\end{align*}
These dominate the contributions from Lemma~\ref{lem:bd-zeta1}.

Computing the constants $\alpha_s \pmod{p^\lambda}$ makes a negligible contribution.
We may simply build Pascal's triangle to height $\lambda + \tau = O(\lambda)$, and use the identity $\binom{-\lambda}t = (-1)^t \binom{\lambda + t}t$.
The time cost is $(\lambda^2) \lambda \log^{1+\eps}(2\lambda) (a \log p)^{1+\eps}$.
\end{proof}

\subsection{The naive algorithm for $Z_X(T)$}

We conclude this section by analysing the complexity of the naive algorithm for computing~$M_s$, i.e., simply expanding $F^{(p-1)s}$ and reading off the appropriate coefficients, and the resulting complexity of the full zeta function computation.

\begin{prop}
\label{prop:Ms-naive}
Let $1 \leq s \leq 2\lambda$.
Then $M_s \pmod{p^\lambda}$ may be computed in time
 \[ 2^{2n} n^{1+\eps} d^{n+\eps} \lambda^{n+1} \log^{2+\eps}(2\lambda) a^{1+\eps} p^n \log^{2+\eps} p. \]
\end{prop}
\begin{proof*}
Using multivariate Kronecker substitution, multiplying two homogeneous polynomials in $(\ZZ_q/p^\lambda \ZZ_q)[x]$ of degree at most~$m$ takes time
 \[ (2m)^n \log^{1+\eps} ((2m)^n) \lambda \log^{1+\eps}(2\lambda) (a \log p)^{1+\eps}. \]
To compute a power~$F^k$ for $k \geq 1$, we may first recursively compute $F^{\lfloor k/2 \rfloor}$, and then use $F^k = (F^{\lfloor k/2 \rfloor})^2$ if~$k$ is even or $F^k = F \cdot (F^{\lfloor k/2 \rfloor})^2$ if~$k$ is odd.
The total cost of computing~$F^k$ is thus bounded by 
 \[ (2kd)^n \log^{1+\eps} ((2kd)^n) \lambda \log^{1+\eps}(2\lambda) (a \log p)^{1+\eps}. \]
Taking $k = (p-1)s$ and recalling the definition of~$M_s$ in Lemma~\ref{lem:M-prod}, we obtain $M_s \pmod{p^\lambda}$ in time
\begin{multline*}
 (4pd\lambda)^n \log^{1+\eps} ((4pd\lambda)^n) \lambda \log^{1+\eps}(2\lambda) (a\log p)^{1+\eps} \\
 = 2^{2n} n^{1+\eps} d^{n+\eps} \lambda^{n+1} \log^{2+\eps}(2\lambda) a^{1+\eps} p^n \log^{2+\eps} p. \qquad \proofbox
\end{multline*}
\end{proof*}

\begin{thm}
\label{thm:hs-naive}
There exists an explicit deterministic algorithm with the following properties.
The input and output is the same as in Theorem~\ref{thm:hs-linear}.
The algorithm has time complexity
 \[ 2^{6n^2 + 13n} n^{3n+3+\eps} (d+1)^{3n^2+6n+\eps} a^{3n+4+\eps} p^n \log^{2+\eps} p. \]
\end{thm}
\begin{proof}
By Proposition~\ref{prop:Ms-naive} we may compute $M_s \pmod{p^\lambda}$ for $s = 1, \ldots, 2\lambda$, with $\lambda = 2na(4d+4)^n$, in time
\begin{multline*}
 2^{2n} n^{1+\eps} d^{n+\eps} \lambda^{n+2} \log^{2+\eps}(2\lambda) a^{1+\eps} p^n \log^{2+\eps} p \\
 \begin{split}
   & = 2^{2n} n^{1+\eps} d^{n+\eps} (2na(4d+4)^n)^{n+2} \log^{2+\eps}(4na(4d+4)^n) a^{1+\eps} p^n \log^{2+\eps} p \\
   & = 2^{2n^2 + 7n} n^{n+5+\eps} (d+1)^{n^2+3n+\eps} a^{n+3+\eps} p^n \log^{2+\eps} p.
 \end{split}
\end{multline*}
Applying Lemma~\ref{lem:bd-zeta2} and taking dominant exponents leads to the indicated bound.
\end{proof}

\section{Recurrences for polynomial powers}
\label{sec:recurrences}

\subsection{Setting up the recurrences}

The following theorem establishes the ``deformation recurrence'' alluded to Section~\ref{sec:intro}.

\begin{thm}
\label{thm:deform}
Let~$R$ be a domain of characteristic zero.
Let $d \geq 1$, $F \in R[x]_d$, and put $G = x_0^d + \cdots + x_n^d \in R[x]_d$.
Let $s \geq 1$ and $h  \geq (d-1)(n+1) + 1$, and let $v \in B_{ds}$ and $w \in B_h$.
For $k \geq 1$ and $H \in R[x]_{kds}$, let $[H]_k$ denote the vector $(H_{kv + w - t})_{t \in B_h}$.

Then there exists a matrix~$Q$ with the following properties.
Its rows and columns are indexed by~$B_h$.
Its entries are linear polynomials in $R[k,\ell]$.
For all $k_0 \geq 1$ and $0 \leq \ell_0 < k_0 s$,
\begin{equation}
\label{eq:deform1}
 [G^{k_0 s - \ell_0 - 1} F^{\ell_0 + 1}]_{k_0} = \frac{1}{d(k_0 s - \ell_0)} Q(k_0, \ell_0) [G^{k_0 s - \ell_0} F^{\ell_0}]_{k_0}.
\end{equation}
In particular, for any $k_0 \geq 1$,
\begin{equation}
\label{eq:deform}
 [F^{k_0 s}]_{k_0 } = \frac{1}{d^{k_0 s} (k_0 s)!} Q(k_0, k_0 s - 1) \cdots Q(k_0, 0) [G^{k_0 s}]_{k_0 }.
\end{equation}
\end{thm}
\begin{proof}
Let $t \in B_h$.
To determine the row of~$Q$ corresponding to~$t$, we must find an expression for $(G^{ks - \ell - 1} F^{\ell + 1})_{kv + w - t}$ in terms of $G^{ks - \ell} F^\ell$.

Since $\deg t = h > (d-1)(n+1)$, by the pigeonhole principle there is some~$i$ such that $t_i \geq d$.
Let $t' = t - (0, \ldots, d, \ldots, 0) \in B_{h - d}$, so that $x^t = x^{t'} x_i^d$.

Consider the differential operator $\partial = x_i \frac{\partial}{\partial x_i}$.
Its effect on a polynomial $H = \sum_u H_u x^u \in R[x]$ is given by $\partial H = \sum_u u_i H_u x^u$.
The product rule implies that
 \[ \partial(G^{ks-\ell} F^\ell) = (ks-\ell)(\partial G)(G^{ks-\ell-1} F^\ell) + \ell (\partial F) (G^{ks-\ell} F^{\ell-1}). \]
Multiplying by~$F$ and rearranging, we obtain
 \[ (ks-\ell) d x_i^d (G^{ks-\ell-1} F^{\ell+1}) = (F\partial - \ell \partial F)(G^{ks-\ell} F^\ell), \]
and thus
\begin{align*}
  d(ks - \ell)(G^{ks-\ell-1} F^{\ell+1})_{kv + w-t} & = ((F\partial - \ell \partial F)(G^{ks-\ell} F^\ell))_{kv + w - t'} \\
    & = \sum_{y \in B_d} (kv_i + w_i - t'_i - (\ell+1)y_i) F_y (G^{ks-\ell} F^\ell)_{kv + w-t'-y}.
\end{align*}
As~$y$ ranges over~$B_d$, $t' + y$ ranges over a subset of~$B_h$.
Thus we may define~$Q$ by
\begin{equation}
\label{eq:defn-Q}
 Q_{t,z} = (kv_i + w_i - t_i' - (\ell+1)(z_i - t'_i)) F_{z-t'}
\end{equation}
for $z \in B_h$, where we take $Q_{t,z} = 0$ if $z - t' \notin B_d$.
This establishes~\eqref{eq:deform1}, and~\eqref{eq:deform} is obtained by iterating~\eqref{eq:deform1} over $\ell_0 = 0, \ldots, k_0 s - 1$.
\end{proof}

We now return to the notation of Theorem~\ref{thm:trace}, and explain how we will use the crucial identity~\eqref{eq:deform} to efficiently compute $M_s \pmod{p^\lambda}$.
Assume that $p \ndivides d$ and let $s \geq 1$.
Let $v \in B_{ds}$ and consider the $v$-th row of~$M_s$.
Let
 \[ h = \max(ds, (d-1)(n+1) + 1). \]
Choose any $z \in B_{h - ds}$, and put $w = v + z \in B_h$.
We will apply Theorem~\ref{thm:deform} with parameters~$s$, $d$, $h$, $v$, $w$, $R = \ZZ_q$ and $k_0 = p - 1$.
We thus obtain a matrix~$Q$, with entries in $\ZZ_q[k, \ell]$, such that
\begin{equation}
\label{eq:Qprod}
 [F^{(p-1) s}]_{p-1} =
  \frac{1}{d^{(p-1)s} ((p-1)s)!} Q(p - 1, (p-1)s - 1) \cdots Q(p - 1, 0) [G^{(p-1)s}]_{p-1}.
\end{equation}
The $v$-th row of~$M_s$ is easily extracted from $[F^{(p-1)s}]_{p-1}$: for any $u \in B_{ds}$ we have
 \[ (M_s)_{v, u} = (F^{(p-1)s})_{pv - u} = (F^{(p-1)s})_{(p-1)v + w - (u + z)}, \]
and this is exactly the $t$-th component of $[F^{(p-1)s}]_{p-1}$ for $t = u + z \in B_h$.

Thus the problem boils down to evaluating the right side of~\eqref{eq:Qprod}, modulo $p^\lambda$.
Let $v_p(\cdot)$ denote the $p$-adic valuation, normalised so that $v_p(p) = 1$.
We have $v_p(d^{(p-1)s}) = 0$ and $v_p(((p-1)s)!) \leq s$, by the well-known estimate $v_p(m!) \leq \frac{m}{p-1}$.
Thus it suffices to compute the denominator
 \[ d^{(p-1)s} ((p-1)s)! \pmod{p^{\lambda_1}}, \]
the ``initial vector''
 \[ [G^{(p-1)s}]_{p-1} \pmod{p^{\lambda_1}}, \]
and the matrix-vector product
\begin{equation}
\label{eq:prod3}
  Q(p - 1, (p-1)s - 1) \cdots Q(p - 1, 0) [G^{(p-1)s}]_{p-1} \pmod{p^{\lambda_1}},
\end{equation}
where $\lambda_1 = \lambda + s$.
The only difference between Theorems~\ref{thm:hs-linear}, \ref{thm:hs-sqrt} and~\ref{thm:hs-avgpoly} is in how we evaluate these products.

Let us examine these products more closely.
In the first one, $d^{(p-1)s} \pmod{p^{\lambda_1}}$ may be computed efficiently using binary powering, and this has negligible complexity compared to the rest of the computation; henceforth we will focus on
\begin{equation}
\label{eq:prod1}
 ((p-1)s)! \pmod{p^{\lambda_1}}.
\end{equation}
For the $[G^{(p-1)s}]_{p-1}$ term, note that its $t$-th component, for $t \in B_h$, is the multinomial coefficient
 \[ (G^{(p-1)s})_{dy} = \binom{(p-1)s}{y_0 \cdots y_n} = \frac{((p-1)s)!}{y_0! \cdots y_n!}, \]
where $y = ((p-1)v + w - t) / d$.
We understand this to be zero if any~$y_i$ is negative or non-integral.
We have $v_p(y_0! \cdots y_n!) \leq \sum_i y_i / (p-1) = s$, so it suffices to compute
\begin{equation}
\label{eq:prod2}
 y_0! \cdots y_n! \pmod{p^{\lambda_2}}
\end{equation}
and $((p-1)s)! \pmod{p^{\lambda_2}}$ where $\lambda_2 = \lambda_1 + s = \lambda + 2s$.

\subsection{Linear time algorithm}

The next result carries out the above plan, using the naive algorithm to evaluate each product.
\begin{prop}
\label{prop:Ms-linear}
Let $\lambda \geq (n + 1)/2$ and $1 \leq s \leq 2\lambda$.
Assume that $p \ndivides d$.
Then $M_s \pmod{p^\lambda}$ may be computed in time
 \[ (n + 2d\lambda)^{3n} \lambda^2 \log^{1+\eps}(2\lambda) a^{1+\eps} p \log^{1+\eps} p \]
and space
 \[ O((n + 2d\lambda)^{2n} \lambda a \log p). \]
\end{prop}
\begin{proof}
To compute the row of $M_s \pmod{p^\lambda}$ corresponding to a given $v \in B_{ds}$, we continue with the notation established above.
Note that $\lambda_1 = O(\lambda)$ and $\lambda_2 = O(\lambda)$.

We compute~\eqref{eq:prod1} by the naive algorithm, i.e., start with~$1$, and successively multiply by $2, 3, \ldots, (p-1)s$, reducing modulo~$p^{\lambda_1}$ after each multiplication.
The time complexity is
 \[ s \lambda_1 \log^{1+\eps}(2\lambda_1) p \log^{1+\eps} p = \lambda^2 \log^{1+\eps}(2\lambda) p \log^{1+\eps} p. \]
The space complexity is only $O(\lambda \log p)$, because we may overwrite the accumulated product as we proceed.

For each $t \in B_h$, the product~\eqref{eq:prod2} is handled similarly.
The number of factors is again $(p-1)s$, so the time complexity is
 \[ |B_h| \lambda^2 \log^{1+\eps}(2\lambda) p \log^{1+\eps} p, \]
and the space complexity is $O(|B_h| \lambda \log p)$.

For~\eqref{eq:prod3}, we must multiply $[G^{(p-1)s}]_{p-1}$ by $Q(p-1, \ell)$ for $\ell = 0, 1, \ldots, (p-1)s - 1$ in turn.
Each matrix $Q(p - 1, \ell) \pmod{p^{\lambda_1}}$ occupies space $O(|B_h|^2 \lambda a \log p)$ and may be computed easily from~\eqref{eq:defn-Q} in time
 \[ |B_h|^2 \lambda \log^{1+\eps}(2\lambda) (a \log p)^{1+\eps}. \]
This is also the time complexity of each matrix-vector product, so the total time over all~$\ell$ is
 \[ |B_h|^2 \lambda^2 \log^{1+\eps}(2\lambda) a^{1+\eps} p \log^{1+\eps} p. \]
We conclude that computing the $v$-th row of $M_s \pmod{p^\lambda}$ may be achieved within the same time bound.
The space may be reused for each matrix-vector product, so the space complexity is $O(|B_h|^2 \lambda a \log p)$.

Repeating the above for each $v \in B_{ds}$, we obtain the whole matrix $M_s \pmod{p^\lambda}$ in time
 \[ |B_{ds}| |B_h|^2 \lambda^2 \log^{1+\eps}(2\lambda) a^{1+\eps} p \log^{1+\eps} p. \]
The space complexity is still $O(|B_h|^2 \lambda a \log p)$, as we may reuse the space for each row, and this is also enough space to store $M_s \pmod{p^\lambda}$ itself.

Finally, observe that $ds \leq 2d\lambda$ and
 \[ (d-1)(n+1) + 1 = dn - n + d \leq d(n+1) \leq 2d\lambda, \]
so also $h \leq 2d\lambda$.
Thus $|B_h| \leq (n + 2d\lambda)^n$ and $|B_{ds}| \leq (n + 2d\lambda)^n$, as in~\eqref{eq:B-bound}.
\end{proof}

\begin{proof}[of Theorem \ref{thm:hs-linear}]
To apply Lemma~\ref{lem:bd-zeta2}, we take $\lambda = 2na(4d+4)^n$ and use Proposition~\ref{prop:Ms-linear} to compute $M_s \pmod{p^\lambda}$ for $1 \leq s \leq 2\lambda$.
The total time is
\begin{multline*}
 (n + 2d\lambda)^{3n} \lambda^3 \log^{1+\eps}(2\lambda) a^{1+\eps} p\log^{1+\eps} p \\
   \begin{split}
   & = (5nda(4d+4)^n)^{3n} (na (4d+4)^n)^3 \log^{1+\eps}(4na(4d+4)^n) a^{1+\eps} p\log^{1+\eps} p \\
   & = 2^{6n^2 + 13n} n^{3n+4+\eps} (d+1)^{3n^2+6n+\eps} a^{3n+4+\eps} p \log^{1+\eps} p,
   \end{split}
\end{multline*}
and the total space is
\begin{align*}
 O((n + 2d\lambda)^{2n} \lambda^2 a \log p) 
 & = O((5dna(4d+4)^n)^{2n} (na(4d+4)^n)^2 a \log p) \\
 & = O(2^{4n^2+9n} n^{2n+2} (d+1)^{2n^2+4n} a^{2n+3} \log p).
\end{align*}
These dominate the contributions from Lemma~\ref{lem:bd-zeta2}.
\end{proof}

\subsection{Square-root time algorithm}

To reduce the time complexity from $p^{1+\eps}$ to $p^{1/2+\eps}$, we will employ the following algorithm of Bostan, Gaudry and Schost.
\begin{lem}
\label{lem:BGS}
Let $m \geq 1$ and $\mu \geq 1$.
Let~$U(\ell)$ be an $m \times m$ matrix whose entries are linear polynomials in $(\ZZ_q/p^\mu \ZZ_q)[\ell]$, and let $Y_0 \in (\ZZ_q/p^\mu \ZZ_q)^m$.
Consider the recurrence $Y_{\ell+1} = U(\ell) Y_\ell$ for $\ell \geq 0$. For any $1 \leq \ell_0 \leq p - 1$, we may compute~$Y_{\ell_0}$ in time
 \[ m^3 \mu \log^{1+\eps}(2\mu) a^{1+\eps} p^{1/2} \log^{2+\eps} p \]
and space
 \[ O(m^2 \mu a p^{1/2} \log p). \]
\end{lem}
\begin{proof}
This is a special case of~\cite[Theorem~14]{BGS-recurrences}.
The invertibility hypothesis of that theorem holds because the integers $1, 2, \ldots, \lfloor \sqrt{p - 1} \rfloor + 1$ are not divisible by $p$ (unless $p = 2$, in which case the lemma is trivial).
We remark that $\mu \log^{1+\eps}(2\mu) a^{1+\eps} p^{1/2} \log^{2+\eps} p$ is the cost of multiplying polynomials of degree $O(p^{1/2})$ over $\ZZ_q/p^\mu \ZZ_q$, and the~$m^3$ term arises from matrix multiplication.
\end{proof}

\begin{prop}
\label{prop:Ms-sqrt}
Let $\lambda \geq (n + 1)/2$ and $1 \leq s \leq 2\lambda$.
Assume that $p \ndivides d$.
Then $M_s \pmod{p^\lambda}$ may be computed in time
 \[ (n + 2d\lambda)^{4n} \lambda^2 \log^{1+\eps}(2\lambda) a^{1+\eps} p^{1/2} \log^{2+\eps} p \]
and space
 \[ O((n + 2d\lambda)^{2n} \lambda a p^{1/2} \log p). \]
\end{prop}
\begin{proof}
We use the same setup as in the proof of Proposition~\ref{prop:Ms-linear}.

To compute~\eqref{eq:prod1}, we apply Lemma~\ref{lem:BGS} with $m = 1$, $\mu = \lambda_1$, $U(\ell) = \ell + 1$, $Y_0 = 1$ and $q = p$.
Then $Y_\ell = \ell!$ for $\ell \geq 0$, and we may compute $(p-1)! \pmod{p^{\lambda_1}}$ in time
 \[ \lambda \log^{1+\eps}(2\lambda) p^{1/2} \log^{2+\eps} p \]
and space $O(\lambda p^{1/2} \log p)$.
Repeating this~$s$ times, and replacing~$Y_0$ by the accumulated product after each invocation, we obtain $((p-1)s)! \pmod{p^{\lambda_1}}$ in time
 \[ \lambda^2 \log^{1+\eps}(2\lambda) p^{1/2} \log^{2+\eps} p. \]
The space may be reused.

(One may save a factor of $s^{1/2}$ in time by treating the whole product of length $O(ps)$ in one pass, at the expense of introducing complications involving the invertibility hypotheses; see for example~\cite[p.~1798]{BGS-recurrences}.)

For~\eqref{eq:prod2} we use the same strategy.
For each $t \in B_h$ we must compute~$n + 1$ factorials of length at most $(p-1)s$.
The time complexity is
 \[ n |B_h| \lambda^2 \log^{1+\eps}(2\lambda) p^{1/2} \log^{2+\eps} p. \]
The space complexity is $O(\lambda p^{1/2} \log p)$, plus $O(|B_h| \lambda \log p)$ to store the output.

Finally, for~\eqref{eq:prod3} we take $m = |B_h|$, $\mu = \lambda_1$, $U(\ell) = Q(p-1, \ell) \pmod{p^{\lambda_1}}$ and $Y_0 = [G^{(p-1)s}]_{p-1}$.
Splitting again into~$s$ subproducts, the time complexity is
 \[ |B_h|^3 \lambda^2 \log^{1+\eps}(2\lambda) a^{1+\eps} p^{1/2} \log^{2+\eps} p, \]
which dominates the contributions from~\eqref{eq:prod1} and~\eqref{eq:prod2}.
The space complexity is
 \[ O(|B_h|^2 \lambda a p^{1/2} \log p). \]
This may be reused for each~$v$, and includes the space required for the final output.

Summing over $v \in B_{ds}$ yields the desired bounds, analogously to the proof of Proposition~\ref{prop:Ms-linear}.
\end{proof}

\begin{proof}[of Theorem~\ref{thm:hs-sqrt}]
Identical to the proof of Theorem~\ref{thm:hs-linear}, using Proposition~\ref{prop:Ms-sqrt} instead of Proposition~\ref{prop:Ms-linear} to compute the $M_s \pmod{p^\lambda}$.
\end{proof}

\subsection{Average polynomial time algorithm}

Next we prove Theorem~\ref{thm:hs-avgpoly}.
The key tool is the following lemma, which is a generalisation of the ``accumulating remainder tree for matrices'' of~\cite[Proposition~4]{Har-avgpoly}.
The main difference is that here we work with matrices whose entries are truncated power series over~$\ZZ$, instead of simply integers.
The proof is otherwise essentially identical.
We also bound the space complexity, which was ignored in~\cite{Har-avgpoly}.

Recall that in Section~\ref{sec:intro}, for $H \in \ZZ[x_0, \ldots, x_n]$ we defined $\norm H = \max_u |H_u|$.
If $\beta \geq 1$ and $h \in \ZZ[k]/k^\beta$, say
 \[ h = h_0 + h_1 k + \cdots + h_{\beta-1} k^{\beta-1} \pmod{k^\beta}, \]
we define $\norm h = \sum_{i=0}^{\beta-1} |h_i|$.
Note that this norm is submultiplicative, i.e.~if $h, h' \in \ZZ[k]/k^\beta$, then $\norm{hh'} \leq \norm h \norm{h'}$.
If~$E$ is a matrix with entries in $\ZZ[k]/k^\beta$, we define $\norm{E} = \max_j \sum_i \norm{E_{ij}}$, i.e., the maximum of the~$L^1$ norms of the columns of~$E$.
This norm satisfies $\norm{E E'} \leq \norm{E} \norm{E'}$ (the proof is easy; see~\cite[\S2]{Har-avgpoly}).

\begin{lem}
\label{lem:avgpoly}
Let $m \geq 1$, $\beta \geq 1$, $\mu \geq 1$, $N \geq 2$, and let $\rho \in \RR$, $\rho > 1$.
We are given as input a sequence of $m \times m$ matrices $E_1, \ldots, E_{N-1}$, with entries in $\ZZ[k]/k^\beta$, such that $\log \norm{E_j} \leq \rho$ for all~$j$.
Then we may compute
 \[ E_{p-1} \cdots E_2 E_1 \pmod{p^\mu} \]
for all primes $p < N$ simultaneously in time
 \[ m^3 \beta (\mu + \rho) N \log N \log^{1+\eps}(\beta \mu \rho N) \]
and space
 \[ O(m^2 \beta (\mu + \rho) N \log N). \]
\end{lem}
\begin{proof*}
We will construct several binary trees of depth $\ell = \lceil \log_2 N \rceil$, with nodes indexed by the pairs $(i, t)$ with $0 \leq i \leq \ell$ and $0 \leq t < 2^i$.
The root node is at level~$i = 0$ and the leaf nodes are at level~$i = \ell$.
The children of $(i, t)$ are $(i+1, 2t)$ and $(i+1, 2t+1)$.
To each node $(i, t)$ we associate the set
 \[ S_{i,t} = \left\{ j \in \ZZ : t \frac{N}{2^i} \leq j < (t+1)\frac{N}{2^i} \right\}. \]
At level~$i$, the sets $S_{i,t}$ partition $\{0, 1, \ldots, N - 1\}$ into~$2^i$ sets of roughly equal size.
At the top level we have $S_{0,0} = \{0, \ldots, N-1\}$.
At level~$\ell$ we have $|S_{\ell,t}| \leq 1$ for every~$t$, and for every $0 \leq j < N$, there is exactly one~$t$, namely $t = \lfloor 2^\ell j / N \rfloor$, such that $S_{\ell,t}= \{j\}$.
For $0 \leq i < \ell$ we have the disjoint union $S_{i,t} = S_{i+1,2t} \cup S_{i+1, 2t+1}$.
We write $P_{i,t}$ for the set of primes in $S_{i,t}$.

The first tree is the \emph{modulus tree}, defined by
 \[ M_{i,t} = \prod_{p \in P_{i,t}} p^\mu \in \ZZ. \]
To compute the modulus tree we use a standard product tree algorithm~\cite{Ber-fastmult}.
We assume that the primes up to~$N$ are known.
For each leaf node either $M_{\ell,t} = 1$, or $M_{\ell,t} = p^\mu$ for an appropriate~$p$.
Starting from the leaf nodes, we repeatedly use the identity $M_{i,t} = M_{i+1,2t} M_{i+1,2t+1}$ to work up to the root.
At each level the total space occupied by the $M_{i,t}$ is $O(\sum_{p < N} \log(p^\mu)) = O(\mu N)$, so each level takes time $\mu N \log^{1+\eps}(\mu N)$.
The time cost over the whole tree is
 \[ \mu N \log N \log^{1+\eps}(\mu N), \]
and the space occupied by the tree is
 \[ O(\mu N \log N). \]

Next we define the \emph{value tree} by
 \[ V_{i,t} = \prod_{j \in S_{i,t}} E_j. \]
Here the~$E_j$ are multiplied in descending order, as in the statement of the lemma, and for convenience we put $E_0 = I$ (the identity matrix).
Again we use a product tree to compute the $V_{i,t}$.
The space occupied by a single~$E_j$ is $O(m^2 \beta \rho)$, and by submultiplicativity, each level requires space $O(m^2 \beta \rho N)$.
The time to compute each level is $m^3 \beta \rho N \log^{1+\eps}(\beta \rho N)$, where the~$m^3$ term arises from matrix multiplication.
The whole tree is computed in time
 \[ m^3 \beta \rho N \log N \log^{1+\eps}(\beta \rho N) \]
and occupies space
 \[ O(m^2 \beta \rho N \log N). \]

Finally we define the \emph{accumulating remainder tree} by
 \[ A_{i,t} = V_{i,t-1} \cdots V_{i,1} V_{i,0} \pmod{M_{i,t}}. \]
The leaf nodes contain the desired output, i.e., for any $p < N$, choosing~$t$ so that $S_{\ell,t} = \{p\}$, we have $P_{\ell,t} = p^\mu$ and $A_{\ell,t} = E_{p-1} \cdots E_1 \pmod{p^\mu}$.
To compute the $A_{i,t}$, we start with $A_{0,0} = I$, and work downwards via the relations
\begin{align*}
  A_{i+1,2t}   & = \phantom{V_{i+1,2t}} A_{i,t} \pmod{M_{i+1, 2t}}, \\
  A_{i+1,2t+1} & =          V_{i+1,2t}  A_{i,t} \pmod{M_{i+1,2t+1}}. 
\end{align*}
Each $A_{i,t}$ occupies space $O(m^2 |P_{i,t}| \beta \mu \log N)$, so the space required at each level is $O(m^2 \beta \mu N)$, and the time for each level is $m^3 \beta (\mu + \rho) N \log^{1+\eps}(\beta \mu \rho N)$.
Over the whole tree, the time cost is
 \[ m^3 \beta (\mu + \rho) N \log N \log^{1+\eps}(\beta \mu \rho N) \]
and the space occupied by the tree is
\multbox
\begin{eqnarray*}
 O(m^2 \beta (\mu + \rho) N \log N).
\end{eqnarray*}
\emultbox
\end{proof*}

Now we return to the setting of Theorem~\ref{thm:hs-avgpoly}.
Let $F \in \ZZ[x]_d$, and for each prime~$p$, let $\red F_p \in \FF_p[x]_d$ be the reduction of~$F$ modulo~$p$, and let~$X_p$ be the corresponding hypersurface in $\TT^n_{\FF_p}$.
For each~$p$ we may apply the results of Section~\ref{sec:trace} to~$\red F_p$ (taking~$a = 1$).
In particular, let $M_{p,s}$ be the matrix, previously denoted by~$M_s$, associated to~$\red F_p$ for each $s \geq 1$.
Then Lemma~\ref{lem:bd-zeta2} shows how to compute $Z_{X_p}(T)$ in terms of $M_{p,s} \pmod{p^\lambda}$ for suitable~$\lambda$ and sufficiently many~$s$.
To prove Theorem~\ref{thm:hs-avgpoly}, we must show, for each~$s$, how to efficiently compute $M_{p,s} \pmod{p^\lambda}$ for all $p < N$ simultaneously.

We will use the same framework discussed after the proof of Theorem~\ref{thm:deform}.
Let $s \geq 1$ and $v \in B_{ds}$.
Let~$h$, $z$ and $w$ be defined as before, and apply Theorem~\ref{thm:deform} with parameters~$s$, $d$, $h$, $v$, $w$, but now with $R = \ZZ$ instead of~$\ZZ_p$.
We obtain a certain matrix~$Q$, with entries in $\ZZ[k, \ell]$, such that
 \[ [F^{(p-1) s}]_{p-1} = \frac{1}{d^{(p-1)s} ((p-1)s)!} Q(p - 1, (p-1)s - 1) \cdots Q(p - 1, 0) [G^{(p-1)s}]_{p-1} \]
for every~$p$.
This is the same as~\eqref{eq:Qprod}, but now takes place over~$\ZZ$ instead of~$\ZZ_p$.
On the other hand, taking the image over~$\ZZ_p$, we see that the entries of $M_{p,s} \pmod{p^\lambda}$ may be extracted from $[F^{(p-1)s}]_{p-1} \pmod{p^\lambda}$ just as before, and then we may apply Lemma~\ref{lem:bd-zeta2} for each prime separately.
(In other words,~$F$ happens to be a lift of~$\red F_p$, for every~$p$.)

\begin{prop}
\label{prop:Ms-avgpoly}
Let $\lambda \geq (n + 1)/2$ and $1 \leq s \leq 2\lambda$.
Then $M_{p,s} \pmod{p^\lambda}$ may be computed simultaneously for all $p < N$, $p \ndivides d$, in time
 \[ (n + 2d\lambda)^{4n} n^{1+\eps} \lambda^2 N \log N \log(\lambda N) \log^{1+\eps}(d\lambda N\norm F) \]
and space
 \[ O((n + 2d\lambda)^{2n} n \lambda^2 N \log N \log(nd \lambda N \norm F)). \]
\end{prop}
\begin{proof}
Continuing the argument above, we must show how to evaluate~\eqref{eq:prod1},~\eqref{eq:prod2} and~\eqref{eq:prod3} for all $p < N$ simultaneously.

For~\eqref{eq:prod1}, we will apply Lemma~\ref{lem:avgpoly} with $m = \beta = 1$, $\mu = \lambda_1$ and
 \[ E_j = (js) (js - 1) \cdots (js - s + 1) \in \ZZ \]
for $1 \leq j < N$, so that $E_{p-1} \cdots E_1 = ((p-1)s)!$.
Note that $\norm{E_j} \leq (sN)^s$, so we may take $\rho = O(s \log(sN)) = O(\lambda \log(\lambda N))$.
Thus we may compute $((p-1)s)! \pmod{p^{\lambda_1}}$ for all $p < N$ in time
 \[ \lambda \log(\lambda N) N \log N \log^{1+\eps}(\lambda^2 \log(\lambda N) N) = \lambda N \log N \log^{2+\eps}(\lambda N) \]
and space
 \[ O(\lambda N \log N \log(\lambda N)).  \]
This also covers the time required to compute the~$E_j$ themselves.

Next consider~\eqref{eq:prod2}.
Let $t \in B_h$, and for $j \geq 1$ and $0 \leq i \leq n$ let
 \[ y_i(j) = \frac{(j-1)v_i + w_i - t_i}{d}. \]
Declare a prime~$p$ to be \emph{relevant} (for this~$t$) if~$y_i(p)$ is non-negative and integral for all~$i$.
We must show how to compute
 \[ y_0(p)! \cdots y_n(p)! \pmod{p^{\lambda_2}}, \]
for all relevant $p < N$.
We set up a recurrence for this as follows.
Define
 \[ \tilde y_i(j) = \max(0, \lfloor y_i(j) \rfloor). \]
Then each~$\tilde y_i$ is a non-decreasing function of~$j$, and for all relevant~$p$ we have $y_i(p) = \tilde y_i(p)$ for each~$i$.
Let
 \[ E_1 = \tilde y_0(2)! \cdots \tilde y_n(2)! \]
and let
 \[ E_j = \frac{\tilde y_0(j+1)! \cdots \tilde y_n(j+1)!}{\tilde y_0(j)! \cdots \tilde y_n(j)!} = \prod_{i=0}^n \left(\prod_{\ell=\tilde y_i(j) + 1}^{\tilde y_i(j+1)} \ell\right) \]
for $2 \leq j < N$. Then
 \[ E_{p-1} \cdots E_1 = \tilde y_0(p)! \cdots \tilde y_n(p)! = y_0(p)! \cdots y_n(p)! \]
for each relevant~$p$.
To estimate the size of~$E_j$, note that each~$\ell$ in the above product satisfies
 \[ \ell \leq \frac{pds + h}{d} \leq ps + 2\lambda = O(\lambda N), \]
and
\begin{align*}
 \sum_{i=0}^n \tilde y_i(j+1) - \tilde y_i(j)
    & \leq \sum_i \left\lfloor \frac{jv_i + w_i - t_i}{d} \right\rfloor - \left\lfloor \frac{(j - 1)v_i + w_i - t_i}{d} \right\rfloor \\
    & \leq \sum_i (1 + v_i/d) = (n + 1) + s = O(\lambda),
\end{align*}
so
 \[ \log |E_j| = O(\lambda \log(\lambda N)) \]
for $2 \leq j < N$.
A similar argument leads to the same bound for $\log |E_1|$.
Applying Lemma~\ref{lem:avgpoly} with $m = \beta = 1$, $\mu = \lambda_2$, $\rho = O(\lambda \log(\lambda N))$, and the~$E_j$ just defined, we obtain $y_0(p)! \cdots y_n(p)! \pmod{p^{\lambda_2}}$ for all (relevant)~$p$ in time
 \[ \lambda N \log N \log^{2+\eps}(\lambda N) \]
and space
 \[ O(\lambda N \log N \log(\lambda N)).  \]
Summing over all $t \in B_h$, we may compute $[G^{(p-1) s}]_{p-1} \pmod{p^{\lambda_1}}$ for all $p < N$ in time
 \[ |B_h| \lambda N \log N \log^{2+\eps}(\lambda N) \]
and space
 \[ O(|B_h| \lambda N \log N \log(\lambda N)). \]

Finally, for~\eqref{eq:prod3} we apply Lemma~\ref{lem:avgpoly} with $m = |B_h|$, $\beta = \mu = \lambda_1$, and
 \[ E_j = Q(k-1, js-1) \cdots Q(k-1, js-s) \pmod{k^{\lambda_1}} \]
for $1 \leq j < N$.
Here we are regarding $Q(k-1, \ell)$, for each $0 \leq \ell < (N-1)s$, as a matrix over $\ZZ[k]/k^{\lambda_1}$.
To estimate $\norm{Q(k-1,\ell)}$, observe that in~\eqref{eq:defn-Q}, the variables~$v_i$, $w_i$, $t_i'$ and $z_i$ are non-negative integers bounded by~$h$, and we have $\ell+1 \leq sN = O(\lambda N)$ and $|F_{z-t'}| \leq \norm{F}$.
Thus each entry of $Q(k-1,\ell)$ has norm in $O(h\lambda N \norm F)$, and so
 \[ \norm{Q(k-1,\ell)} = O(|B_h|h\lambda N \norm F) = O((n + 2d\lambda)^n d \lambda^2 N \norm F). \]
Therefore
 \[ \log\norm{E_j} = O(\lambda \log((n + 2d\lambda)^n d \lambda^2 N \norm F)) = O(n \lambda \log(nd \lambda N \norm F)). \]
By Lemma~\ref{lem:avgpoly} we obtain
\begin{equation}
\label{eq:E-prod}
 E_{p-1} \cdots E_1 = Q(k-1, (p-1)s-1) \cdots Q(k-1, 0) \pmod{k^{\lambda_1}, p^{\lambda_1}}
\end{equation}
for all $p < N$ in time
\begin{multline*}
 |B_h|^3 \lambda (n\lambda \log(nd\lambda N \norm F)) N \log N \log^{1+\eps}(\lambda^2 n \lambda \log(nd\lambda N\norm F)N) \\
    = |B_h|^3 n^{1+\eps} \lambda^2 N \log N \log(\lambda N) \log^{1+\eps}(d\lambda N\norm F)
\end{multline*}
and space
 \[ O(|B_h|^2 n \lambda^2 N \log N \log(nd \lambda N \norm F)). \]
Then, for each $p < N$ separately, we may substitute~$p$ for~$k$ in~\eqref{eq:E-prod}, to obtain the desired matrix products.
The cost of the substitution is negligible.

All of the above must be repeated for each $v \in B_{ds}$; the time bound follows immediately.
For the space bound, observe that the space for each~$v$ may be reused, and that the space required for the output is only $O(|B_{ds}|^2 \lambda N$).
\end{proof}

\begin{proof*}[of Theorem~\ref{thm:hs-avgpoly}]
This is the same as the proof of Theorem~\ref{thm:hs-linear}, but we work on all~$p$ simultaneously.
Take $\lambda = 2n(4d+4)^n$ and use Proposition~\ref{prop:Ms-avgpoly} to compute $M_{p,s} \pmod{p^\lambda}$ for $1 \leq s \leq 2\lambda$ and all $p < N$, $p \ndivides d$.
The time complexity is
\begin{multline*}
 (n + 2d\lambda)^{4n} n^{1+\eps} \lambda^3 N \log N \log(\lambda N) \log^{1+\eps}(d\lambda N \norm F) \\
 \begin{split}
   & = (5nd(4d+4)^n)^{4n} n^{1+\eps} (n(4d+4)^n)^3 N \log N \\ & \qquad \qquad \log(n(4d+4)^n N) \log^{1+\eps}(nd(4d+4)^n N \norm F) \\
   & = 2^{8n^2 + 16n} n^{4n+6+\eps} (d+1)^{4n^2+7n+\eps} N \log^2 N \log^{1+\eps}(N \norm F),
 \end{split}
\end{multline*}
and the space complexity is
\begin{multline*}
 O((n + 2d\lambda)^{2n} n \lambda^3 N \log N \log(nd\lambda N \norm F)) \\
 \begin{split}
  & = O((5nd(4d+4)^n)^{2n} n (n(4d+4)^n)^3 N \log N \log(ndN\norm F)) \\
  & = O(2^{4n^2+11n} n^{2n+4} (d+1)^{2n^2+5n} N \log N \log(ndN \norm F)).
 \end{split}
\end{multline*}
Then we apply Lemma~\ref{lem:bd-zeta2} separately for each~$p$.
The time for this step is only
 \[ 2^{6n^2 + 13n} n^{3n+3+\eps} (d+1)^{3n^2+6n+\eps} N \log^{1+\eps} N, \]
and the space complexity is only
\multbox
\begin{eqnarray*}
 O(2^{4n^2 + 9n} n^{2n+2} (d+1)^{2n^2+4n} N).
\end{eqnarray*}
\emultbox
\end{proof*}

\section*{Acknowledgments}

Feedback from Alan Lauder led to considerable simplification of the statement and proof of the trace formula.
The author thanks Daniel Chan and Jesse Kass for helpful discussions on respectively arithmetic schemes and algebraic curves, Wouter Castryck, Alan Lauder and Andrew Sutherland for comments on a draft of the paper, and the referee for their suggestions.
The author was supported by the Australian Research Council, DECRA Grant DE120101293.

\bibliographystyle{amsalpha}
\bibliography{arithzeta}

\end{document}